\documentclass[11pt,a4paper,british]{amsart}
\usepackage[defs,cref,semib,href,amsfix]{tms}

\author[T.~Mettler]{Thomas Mettler}
\address{Institut f\"ur Mathematik, Goethe-Universit\"at Frankfurt, 60325 Frankfurt am Main, Germany}
\email{mettler@math.uni-frankfurt.de, mettler@math.ch}
\title[metrisability of projective surfaces]{Metrisability of Projective Surfaces and Pseudo-Holomorphic Curves}
\newcommand{\pdfname}{2cQbhrKA}

\renewcommand{\frame}{F}
\newcommand{\twist}{Z}

\renewcommand{\Re}{\operatorname{Re}}

\newcommand{\phispace}{\Gamma(V_0)}
\newcommand{\dualbeta}{B}

\usepackage{underscore}

\date{July 7, 2020}

\begin{document}

\begin{abstract}
We show that the metrisability of an oriented projective surface is equivalent to the existence of pseudo-holomorphic curves. A projective structure $\mathfrak{p}$ and a volume form $\sigma$ on an oriented surface $M$ equip the total space of a certain disk bundle $\twist \to M$ with a pair $(J_{\mathfrak{p}},\mathfrak{J}_{\mathfrak{p},\sigma})$ of almost complex structures. A conformal structure on $M$ corresponds to a section of $\twist \to M$ and $\mathfrak{p}$ is metrisable by the metric $g$ if and only if $[g] : M \to \twist$ is a pseudo-holomorphic curve with respect to $J_{\mathfrak{p}}$ and $\mathfrak{J}_{\mathfrak{p},dA_g}$. 
\end{abstract}

\maketitle

\section{Introduction}

A~\textit{projective structure} on a smooth manifold consists of an equivalence class $\mathfrak{p}$ of torsion-free connections on its tangent bundle, where two such connections are called equivalent if they have the same geodesics up to parametrisation. A projective structure $\mathfrak{p}$ is called~\textit{metrisable} if it contains the Levi-Civita connection of some Riemannian metric. The problem of (locally) characterizing the projective structures that are metrisable was first studied in the work of R.~Liouville~\cite{21.0317.01} in 1889, but was solved only relatively recently by Bryant, Dunajski and Eastwood for the case of two dimensions~\cite{MR2581355}. Since then, there has been renewed interest in the problem, see~\cite{MR3158041,MR3763810,MR3536153,arXiv:1712.06191,MR2384718,MR3969434,MR3159950,MR2876789,arXiv:1812.03591} for related recent work.

The purpose of this short note is to show that in the case of an oriented projective surface $(M,\mathfrak{p})$, the metrisability of $\mathfrak{p}$ is equivalent to the existence of certain pseudo-holomorphic curves. 

An orientation compatible complex structure on $M$ corresponds to a section of the bundle $\pi : \twist \to M$ whose fibre at $x \in M$ consists of the orientation compatible linear complex structures on $T_xM$. The choice of a torsion-free connection $\nabla$ on $TM$ equips $\twist$ with an almost complex structure $J$~\cite{MR728412,MR812312}. Namely, at $j \in \twist$ we lift $j$ horizontally and take a natural complex structure on each fibre vertically. It turns out that $J$ is always integrable and does only depend on the projective equivalence class $\mathfrak{p}$ of $\nabla$, we thus denote it by $J_{\mathfrak{p}}$. Reversing the orientation on each fibre yields another almost complex structure $\mathfrak{J}$ which is however never integrable and is not projectively invariant. Fixing a volume form $\sigma$ on the projective surface $(M,\mathfrak{p})$ determines a unique representative connection ${}^{\sigma}\nabla \in \mathfrak{p}$ which preserves $\sigma$. We will write $\mathfrak{J}_{\mathfrak{p},\sigma}$ for the non-integrable almost complex structure arising from ${}^{\sigma}\nabla \in \mathfrak{p}$. 

The choice of a conformal structure $[g]$ on an oriented surface $M$ defines an orientation compatible complex structure by rotating a tangent vector counterclockwise by $\pi/2$ with respect to $[g]$. Thus, we may think of a conformal structure as a section $[g] : M \to \twist$. Denoting the area form of a Riemannian metric $g$ by $dA_g$, we show:
\begin{thm}\label{mainthm:pseudohol}
An oriented projective surface $(M,\mathfrak{p})$ is metrisable by the metric $g$ on $M$ if and only if $[g] : M \to (\twist,J_{\mathfrak{p}})$ is a holomorphic curve and $[g] : M \to (\twist,\mathfrak{J}_{\mathfrak{p},dA_g})$ is a pseudo-holomorphic curve. 
\end{thm}
Applying a general existence result for pseudo-holomorphic curves~\cite[Theorem III]{MR0149505} it follows that locally we can always find a Riemannian metric $g$ so that $[g] : M \to (\twist,J_{\mathfrak{p}})$ is a holomorphic curve or so that $[g] : M \to (\twist,\mathfrak{J}_{\mathfrak{p},dA_g})$ is a pseudo-holomorphic curve. The geometric significance of the existence of such (pseudo-)holomorphic curves is given in \cref{ppn:pseudohol} below. 

The construction of the (integrable) almost complex structure $J_{\mathfrak{p}}$ on $Z$ given in~\cite{MR728412,MR812312} is adapted from the construction of an almost complex structure $J$ on the twistor space $Y \to N$ of an oriented Riemannian $4$-manifold $(N,g)$, see~\cite{MR506229}. In the Riemannian setting the almost complex structure $J$ is integrable if and only if $g$ is self-dual. In~\cite{MR848842}, Eells--Salamon observe that reversing the orientation on each fibre of $Y \to N$ associates another almost complex structure $\mathfrak{J}$ on $Y$ to $(N,g)$ which is never integrable. Thus, the non-integrable almost complex structure $\mathfrak{J}$ used here may be thought of as the affine analogue of the non-integrable almost complex structure in oriented Riemannian $4$-manifold geometry.   

\subsection*{Acknowledgements} The author is grateful to Maciej Dunajski and Gabriel Paternain for helpful conversations and correspondence. A part of the research for this article was carried out while the author was visiting FIM at ETH Z\"urich. The author would like to thank FIM for its hospitality and DFG for partial support via the priority programme SPP 2026 ``Geometry at Infinity''.

\section{Pseudo-Holomorphic Curves and Metrisability}

Recall that the set of torsion-free connections on a surface $M$ is an affine space modelled on the smooth sections of the vector bundle $V=S^2(T^*M)\otimes TM$. We have a natural trace mapping $\tr : V \to T^*M$, given in abstract index notation by $A^i_{jk} \mapsto A^{k}_{ik}$, as well as an inclusion $\mathrm{Sym} : T^*M \to V$, given by $b_i \mapsto \delta^i_j b_k+\delta^i_kb_j$. The bundle $V$ thus decomposes as $V=V_0\oplus T^*M$, where $V_0$ denotes the trace-free part of $V$. We have (Cartan, Eisenhart, Weyl) -- the reader may also consult~\cite{MR2384705} for a modern reference:
\begin{lem}
Two torsion-free connections $\nabla$ and $\nabla^{\prime}$ on $TM$ are projectively equivalent if and only if there exists a $1$-form $\xi$ on $M$ so that $\nabla-\nabla^{\prime}=\mathrm{Sym}(\xi)$.
\end{lem} 
This gives immediately: 
\begin{lem}
Let $(M,\mathfrak{p})$ be an oriented projective surface and $\sigma$ a volume form on $M$. Then there exists a unique representative connection ${}^{\sigma}\nabla \in \mathfrak{p}$ preserving $\sigma$. 
\end{lem}
\begin{proof}
Let $\nabla \in \mathfrak{p}$ be a representative connection. Since $\sigma$ is a volume form there exists a unique $1$-form $\alpha$ on $M$ such that $\nabla \sigma=\alpha\otimes \sigma$. An elementary computation shows that the connection $\nabla+\mathrm{Sym}(\xi)$ satisfies
\[
\left(\nabla+\mathrm{Sym}(\xi)\right)\sigma=\nabla\sigma-3\xi\otimes \sigma, 
\]
for all $\xi \in \Omega^1(M)$. Thus the connection ${}^{\sigma}\nabla=\nabla+\frac{1}{3}\mathrm{Sym}(\alpha)$ preserves $\sigma$ and clearly is the only connection in $\mathfrak{p}$ doing so. 
\end{proof}

We also have: 

\begin{lem}\label{Lemma:stillpresvolume}
Let $\varphi\in\phispace$ and $\nabla$ be a torsion-free connection on $TM$. Then $\nabla+\varphi$ preserves a volume form $\sigma$ on $M$  if and only if $\nabla$ preserves the volume form $\sigma$. 
\end{lem}

\begin{proof}
Since $\varphi \in \Gamma(V_0)$, an elementary computation shows that the connections $\nabla$ and $\nabla+\varphi$ induce the same connection on the bundle $\Lambda^2(T^*M)$ whose non-vanishing sections are the volume forms. 
\end{proof}

For our purposes it is convenient to construct the almost complex structures $(J,\mathfrak{J})$ associated to $\nabla$ in terms of the connection form $\theta$ on the oriented frame bundle of $M$. The oriented frame bundle $\frame$ of the oriented surface $M$ is the bundle $\upsilon : \frame \to M$ whose fibre at $x \in M$ consists of the linear isomorphisms $u : \R^2 \to T_xM$ that are orientation preserving with respect to the standard orientation on $\R^2$ and the given orientation on $T_xM$. The group $\mathrm{GL}^+(2,\R)$ acts transitively from the right on each fibre by the rule $R_a(u)=u\circ a$ for all $a \in \mathrm{GL}^+(2,\R), u \in \frame$ and this action turns $\upsilon : \frame \to M$ into a principal right $\mathrm{GL}^+(2,\R)$-bundle. The total space $\frame$ carries a tautological $\R^2$-valued $1$-form $\omega$ defined by $\omega_u=u^{-1}\circ \upsilon^{\prime}_{u}$ and $\omega$ satisfies the equivariance property 
\begin{equation}\label{eq:equiomega}
R_a^*\omega=a^{-1}\omega
\end{equation} 
for all $a \in \mathrm{GL}^+(2,\R)$. We may embed $\mathrm{GL}(1,\C)$ as the subgroup of $\mathrm{GL}^+(2,\R)$ consisting of matrices that commute with the standard linear complex structure on $\R^2$. Note that may think of the oriented frame bundle $\upsilon : \frame \to M$ as a principal $\mathrm{GL}(1,\C)$-bundle over $\twist=\frame/\mathrm{GL}(1,\C)$. We may describe an almost complex structure on $\twist$ by describing the pullback of its $(1,\!0)$-forms to $\frame$. The pullback of a $1$-form on $\twist$ to $\frame$ is~\textit{semi-basic} for the projection $\nu : \frame \to \twist$, that is, it vanishes when evaluated on vector fields that are tangent to the fibres of $\nu$. For $y \in \mathfrak{gl}(2,\R)$ we denote by $Y_y$ the vector field on $\frame$ that is generated by the flow $R_{\exp(ty)}$. Clearly, the vector fields $Y_y$ for $y \in \mathfrak{gl}(1,\C)$ span the vector fields on $\frame$ that are tangent to the fibres of $\nu$. 

Let $\nabla$ be a torsion-free connection on $TM$ with connection form $\theta=(\theta^i_j)$ on $\frame$. Recall that $\theta$ satisfies the equivariance property
\begin{equation}\label{eq:equieta}
R_a^*\theta=a^{-1}\theta a
\end{equation}
for all $a \in \mathrm{GL}^+(2,\R)$ and the structure equations
\begin{align}
\begin{split}\label{eq:struceq}
\d\omega^i&=-\theta^i_j\wedge\omega^j,\\
\d\theta^i_j&=-\theta^i_k\wedge\theta^k_j+\Theta^i_j,
\end{split}
\end{align}
where $\Theta=(\Theta^i_j)$ denotes the curvature form of $\theta$. Since $\theta$ is a principal connection on $\frame$ it also satisfies $\theta(Y_y)=y$ for all $y \in \mathfrak{gl}(2,\R)$.  
Since the Lie algebra of $\mathrm{GL}(1,\C)$ is spanned by the matrices of the form
\[
\begin{pmatrix} z & -w \\ w & z\end{pmatrix}
\]
for $(z,w)\in \R^2$, the complex-valued $1$-forms on $\frame$ that are semi-basic for the projection $\nu : \frame \to \twist$ are spanned by the forms $\omega=\omega^1+\i\omega^2$ and 
\[
\zeta=(\theta^1_1-\theta^2_2)+\i \left(\theta^1_2+\theta^2_1\right)
\]
and their complex conjugates. We now have:

\begin{ppn}\label{Proposition:exalmostcplx}
Let $\nabla$ be a torsion-free connection on $TM$ with connection form $\theta=(\theta^i_j)$ on $\frame$. Then there exists a unique pair $(J,\mathfrak{J})$ of almost complex structures on $\twist$ whose $(1,\! 0)$-forms pull back to become linear combinations of the forms $(\omega,\zeta)$ in the case of $J$ and to $(\omega,\ov{\zeta})$ in the case of $\mathfrak{J}$. Moreover, the almost complex structure $J$ is always integrable, whereas $\mathfrak{J}$ is never integrable. 
\end{ppn}

\begin{proof}
Writing
\[
r\mathrm{e}^{\i\phi}\simeq \begin{pmatrix} r\cos \phi & -r\sin\phi \\ r\sin\phi & r\cos\phi\end{pmatrix}
\]
for the elements of $\mathrm{GL}(1,\C)$, the equivariance property~\eqref{eq:equiomega} of $\omega$ and~\eqref{eq:equieta} of $\theta$ implies
\begin{equation}\label{eq:equionezero}
(R_{r\mathrm{e}^{\i\phi}})^*\omega=\frac{1}{r}\mathrm{e}^{\i\phi}\omega\quad\text{and}\quad (R_{r\mathrm{e}^{\i\phi}})^*\zeta=\mathrm{e}^{-2\i\phi}\zeta.
\end{equation}
It follows that there exists a unique almost complex structure $J$ on $\twist$ whose $(1,\! 0)$-forms pull back to $\frame$ to become linear combinations of the forms $\omega,\zeta$. Likewise there exists a unique almost complex structure $\mathfrak{J}$ on $\twist$ whose $(1,\! 0)$-forms pull back to $F$ to become linear combinations of the forms $\omega,\ov{\zeta}$. Furthermore, simple computations using the structure equations~\eqref{eq:struceq} imply that
\[
0=\d\zeta\wedge\omega\wedge\zeta=\d\omega\wedge\omega\wedge\zeta.
\]
Consequently, the Newlander-Nirenberg theorem~\cite{MR0088770} implies that $J$ is integrable. On the other hand, we get
\[
\d\omega\wedge\omega\wedge\ov{\zeta}=\frac{1}{2}\omega\wedge\ov{\omega}\wedge\zeta\wedge\ov{\zeta}
\]
so that $\mathfrak{J}$ is never integrable. 
\end{proof}
\begin{rmk}\label{rmk:hinvariance}
The equivariance properties~\eqref{eq:equionezero} imply that the bundles
\[
H=\nu^{\prime}\left\{\Re(\zeta)=0,\Im(\zeta)=0\right\}\quad \text{and}\quad V=\nu^{\prime} \{\Re(\omega)=0,\Im(\omega)=0\}
\]
are well-defined distributions on $\twist$ that are invariant with respect to $J$ (and $\mathfrak{J}$). Hence we have $T\twist=H\oplus V$. 
\end{rmk}
For the convenience of the reader, we also show~\cite{MR728412,MR812312}:
\begin{ppn}
Suppose the torsion-free connections $\nabla$ and $\nabla^{\prime}$ on $TM$ are projectively equivalent, then they induce the same integrable almost complex structure $J$ on $\twist$.  
\end{ppn}
\begin{proof}
The connections $\nabla$ and $\nabla^{\prime}$ are projectively equivalent if and only if there exists a $1$-form $\xi$ on $M$ such that $\nabla^{\prime}=\nabla+\mathrm{Sym}(\xi)$. Writing $\theta=(\theta^i_j)$ for the connection form of $\nabla$ on $\frame$ and $\upsilon^*\xi=x_i\omega^i$ for real-valued functions $x_i$ on $\frame$, the connection form $\theta^{\prime}$ of $\nabla^{\prime}$ becomes
\[
\theta^{\prime}=\theta+\begin{pmatrix} 2x_1\omega^1+x_2\omega^2 & x_2\omega^1 \\ x_1\omega^2 & x_1\omega^1+2x_2\omega^2\end{pmatrix}.
\]
Consequently, we obtain 
\[
\zeta^{\prime}=\zeta+(x_1\omega^1-x_2\omega^2)+\i(x_2\omega^1+x_1\omega^2)=\zeta+(x_1+\i x_2)\omega
\]
which shows that the complex span of $\omega,\zeta$ is the same as the one of $\omega,\zeta^{\prime}$ and hence the two integrable almost complex structures are the same. 
\end{proof}
\begin{rmk}
For a projective structure $\mathfrak{p}$ on $M$ we will write $J_{\mathfrak{p}}$ for the integrable almost complex structure defined by any representative connection $\nabla \in \mathfrak{p}$. For a projective structure $\mathfrak{p}$ and a volume form $\sigma$ on $M$ we will write $\mathfrak{J}_{\mathfrak{p},\sigma}$ for the non-integrable almost complex structure defined by the representative connection ${}^{\sigma}\nabla \in \mathfrak{p}$. Note that the non-integrable almost complex structure is not projectively invariant. 
\end{rmk}

Recall that a Weyl connection for a conformal structure $[g]$ is a torsion-free connection ${}^{[g]}\nabla$ on $TM$ which preserves $[g]$. Fixing a Riemannian metric $g \in [g]$, the Weyl connections for $[g]$ can be written as ${}^{[g]}\nabla={}^g\nabla+g\otimes \dualbeta-\mathrm{Sym}(\beta)$
for some $1$-form $\beta$ on $M$ and where $\dualbeta$ denotes the $g$-dual vector field to $\beta$. In~\cite{arXiv:1510.01043} and in the language of thermostats in~\cite{MR4109900}, it was observed that for every choice of a conformal structure $[g]$ on a projective surface $(M,\mathfrak{p})$, there exists a unique Weyl connection ${}^{[g]}\nabla$ for $[g]$ and a unique $1$-form $\varphi \in \phispace$ so that ${}^{[g]}\nabla+\varphi$ is a representative connection of $\mathfrak{p}$. Moreover the endomorphism $\varphi(X)$ is symmetric with respect to $[g]$ for every vector field $X$ on $M$. We call ${}^{[g]}\nabla$ the Weyl connection determined by $[g]$. Explicitly, if $\nabla$ is any representative connection of $\mathfrak{p}$, $g \in [g]$ and if we define a vector field $\dualbeta=\frac{3}{4}\mathrm{tr}\left(g^{\sharp}\otimes (\nabla-{}^g\nabla)_0\right)$,
then
\[
\varphi=\left(\nabla -{}^g\nabla -g \otimes \dualbeta\right)_0\qquad\text{and}\qquad {}^{[g]}\nabla={}^g\nabla+g\otimes \dualbeta-\mathrm{Sym}(\beta),
\]
where $A_0$ denotes the trace-free part of a tensor field $A \in \Gamma(S^2(T^*M)\otimes TM)$. We refer the reader to~\cite{arXiv:1510.01043,MR4109900} for a proof that ${}^{[g]}\nabla$ and $\varphi$ do satisfy the claimed properties.

\begin{ppn}\label{ppn:pseudohol}
Let $(M,\mathfrak{p})$ be an oriented projective surface and $g$ a Riemannian metric on $M$. Then we have:
\begin{itemize}
\item[(i)] $\mathfrak{p}$ contains a Weyl connection for $[g]$ if and only if $[g] : M \to (\twist,J_\mathfrak{p})$ is a holomorphic curve;
\item[(ii)] the Weyl connection determined by $[g]$ is the Levi-Civita connection of $g$ if and only if $[g] : M \to (\twist,\mathfrak{J}_{\mathfrak{p},dA_g})$ is a pseudo-holomorphic curve.
\end{itemize}  
\end{ppn}
\begin{rmk}\label{rmk:defpseudohol}
Here we say $[g] : M \to (Z,\mathfrak{J})$ is a (pseudo-)holomorphic curve if the image $\Sigma=[g](M)\subset Z$ admits the structure of a (pseudo-)holomorphic curve. By admitting the structure of (pseudo-)holomorphic curve, we mean that $\Sigma$ can be equipped with a complex structure $J$, so that the inclusion $\iota : \Sigma \to Z$ is $(J,\mathfrak{J})$-linear, that is, satisfies $\mathfrak{J}\circ \iota^{\prime}=\iota^{\prime} \circ J$. 
\end{rmk}

As an immediate consequence, we obtain the \cref{mainthm:pseudohol}:
\begin{proof}[Proof of \cref{mainthm:pseudohol}]
The projective structure $\mathfrak{p}$ is metrisable by $g$ if and only if the Weyl connection determined by $[g]$ is the Levi-Civita connection of $g$ and the $1$-form $\varphi$ vanishes identically. The claim follows by applying \cref{ppn:pseudohol}. 
\end{proof}
For the proof of \cref{ppn:pseudohol} we also need the following Lemma:
\begin{lem}\label{Lemma:holcurvlemma}
Let $(Z,\mathfrak{J})$ be an almost complex four-manifold and $\omega,\chi \in \Omega^1(Z,\C)$ a basis for the $(1,\! 0)$-forms of $Z$. Suppose $\iota : \Sigma \to Z$ is an immersed surface so that $\iota^*(\omega\wedge\ov{\omega})$ is non-vanishing on $\Sigma$. Then $\Sigma$ admits the structure of a pseudo-holomorphic curve if and only if $\iota^*(\omega\wedge\chi)$ vanishes identically on $\Sigma$.
\end{lem}
\begin{proof}
Since $\iota^*(\omega\wedge\ov{\omega})$ is non-vanishing on $\Sigma$, the forms $\iota^*\omega$ and $\iota^*\ov{\omega}$ span the complex-valued $1$-forms on $\Sigma$. Recall that  $\iota : \Sigma \to Z$ is $(j,\mathfrak{J})$-linear if and only if the pullback of every $(1,\! 0)$-form on $Z$ is a $(1,\! 0)$-form on $\Sigma$, the claim follows. 
\end{proof}

\begin{proof}[Proof of \cref{ppn:pseudohol}]
Let $g$ be a Riemannian metric on the oriented projective surface $(M,\mathfrak{p})$. Without losing generality we can assume that the projective structure $\mathfrak{p}$ arises from a connection of the form ${}^{[g]}\nabla+\varphi$. 
The Weyl connection ${}^{[g]}\nabla$ satisfies 
\[
{}^{[g]}\nabla dA_g=2\beta\otimes dA_g
\]
for some $1$-form $\beta$ on $M$ and hence can be written as ${}^{[g]}\nabla={}^g\nabla+g\otimes \beta^{\sharp}-\mathrm{Sym}(\beta)$.

Now suppose $\nabla \in \mathfrak{p}$ preserves the volume form $dA_g$ of $g$. Then, by \cref{Lemma:stillpresvolume} it must be of the form
\begin{equation}\label{eq:volcon}
\nabla={}^{[g]}\nabla+\varphi+\frac{2}{3}\mathrm{Sym}(\beta)={}^g\nabla+g\otimes \beta^{\sharp}-\frac{1}{3}\mathrm{Sym}(\beta)+\varphi.
\end{equation}
\cref{Proposition:exalmostcplx} and \cref{Lemma:holcurvlemma} imply that the condition that $[g] : M \to \twist$ defines a pseudo-holomorphic curve with respect to $J_{\mathfrak{p}}$ respectively $\mathfrak{J}_{\mathfrak{p},dA_g}$ is equivalent to the condition that on the pullback bundle $[g]^*F \to M$ the form $\omega\wedge \zeta$, respectively $\omega\wedge\ov{\zeta}$ vanishes identically, where $\zeta$ is computed from the connection form of $\nabla$ and where we think of $F$ as fibering over $Z$. Keeping this in mind we now compute the pullback of the forms $\zeta$ and $\ov{\zeta}$ to $[g]^*F$. Recall that the semi-basic $1$-forms on $F$ are spanned by the components of $\omega$, hence there exist unique real-valued functions $g_{ij}=g_{ji}$ on $F$ so that $\upsilon^*g=g_{ij}\omega^i\otimes \omega^j$. Likewise, there exist unique real-valued functions $b_i$ on $F$ so that $\upsilon^*\beta=b_i\omega^i$ and unique real-valued function $A^i_{jk}=A^i_{kj}$ on $F$ so that $(\upsilon^*\varphi)^i_j=A^i_{jk}\omega^k$. The functions $A^i_{jk}$ satisfy furthermore $A^k_{ki}=0$ and $g_{ik}A^k_{jl}=g_{jk}A^k_{il}$ since $\varphi$ takes values in the endomorphisms of $TM$ that are trace-free and symmetric with respect to $g$. The Levi-Civita connection $(\psi^i_j)$ of $g$ is the unique principal $\mathrm{GL}^+(2,\R)$-connection on $F$ that satisfies
\begin{align*}
\d\omega^i&=-\psi^i_j\wedge\omega^j,\\
\d g_{ij}&=g_{ik}\psi^k_j+g_{kj}\psi^k_i.
\end{align*}
The pullback bundle $P:=[g]^*F$ is cut out by the equations $g_{11}=g_{22}$ and $g_{12}=0$. On $P$ we have
\begin{align*}
0=\d g_{12}&=g_{11}\psi^1_2+g_{22}\psi^2_1=g_{11}(\psi^1_2+\psi^2_1),\\
0=\d g_{11}-\d g_{22}&=2g_{11}\psi^1_1-2g_{22}\psi^2_2=g_{11}(\psi^1_1-\psi^2_2)
\end{align*}
On $P$ the condition $g_{ik}A^k_{jl}=g_{jk}A^k_{il}$ implies $A^2_{11}=-A^2_{22}$ and $A^1_{22}=-A^1_{11}$. Writing $A^1_{11}=a_1$ and $A^2_{22}=a_2$ and using~\eqref{eq:volcon}, the connection form $\theta$ of $\nabla$ thus becomes
\begin{multline*}
\theta=\begin{pmatrix} \psi^1_1 & -\psi^2_1 \\ \psi^2_1 & \psi^1_1\end{pmatrix}+\begin{pmatrix} b_1\omega^1 & b_1\omega^2 \\ b_2\omega^1 & b_2\omega^2\end{pmatrix}
-\frac{1}{3}\begin{pmatrix} 2b_1\omega^1+b_2\omega^2 & b_2\omega^1 \\ b_1\omega^2 & b_1\omega^1+2b_2\omega^2\end{pmatrix}\\
+\begin{pmatrix} a_1\omega^1-a_2\omega^2 & -a_2\omega^1-a_1\omega^2 \\ -a_2\omega^1-a_1\omega^2 & -a_1\omega^1+a_2\omega^2\end{pmatrix} 
\end{multline*}
Introducing the complex notation $a=a_1+\i a_2$ and $b=\frac{1}{2}(b_1-\i b_2)$, we obtain from a simple calculation
\[
\zeta=(\theta^1_1-\theta^2_2)+\i(\theta^1_2+\theta^2_1)=\frac{4}{3}\ov{b}\omega+2\ov{a}\ov{\omega},
\] 
where we write $\omega=\omega^1+\i\omega^2$. 

Finally, since $[g] : M \to (\twist,J_{\mathfrak{p}})$ is a holomorphic curve if and only if $\omega\wedge\zeta$ vanishes identically on $P$, it follows that $[g] : M \to (\twist,J_{\mathfrak{p}})$ is a holomorphic curve if and only if 
\[
0=\omega\wedge\zeta=2\ov{a}\omega\wedge\ov{\omega}
\]
which is equivalent to $\varphi$ vanishing identically. This shows (i). 

Likewise $[g] : M \to (\twist,\mathfrak{J}_{\mathfrak{p},dA_g})$ is a pseudo-holomorphic curve if and only if 
\[
0=\omega\wedge\ov{\zeta}=\frac{4}{3}b\omega\wedge\ov{\omega}
\]
on $P$. This is equivalent to $\beta$ vanishing identically. This shows (ii). 
\end{proof}
As a corollary we obtain:
\begin{cor}\label{Corollary:localexistence}
Let $(M,\mathfrak{p})$ be a projective surface. Then locally $\mathfrak{p}$ contains
\begin{itemize}
\item[(i)] a Weyl connection ${}^{[g]}\nabla$ for some conformal structure $[g]$;
\item[(ii)] a connection of the form ${}^{\tilde{g}}\nabla+\varphi$ for some Riemannian metric $\tilde{g}$ and some $\varphi \in \phispace$ with $\varphi$ taking values in the endomorphisms that are $\tilde{g}$-symmetric. 
\end{itemize}
\end{cor}
\begin{rmk}
The first statement of \cref{ppn:pseudohol} and \cref{Corollary:localexistence} was previously obtained in~\cite{MR3144212}.
\end{rmk}

\begin{proof}[Proof of \cref{Corollary:localexistence}]
We first consider the case (ii). We fix a volume form $\sigma$ on $M$. We need to show that in a neighbourhood $U_x$ of every point $x \in M$ there exists a conformal structure $[g]$ which is a pseudo-holomorphic curve into the total space of the bundle $\pi : Z \to M$, where we equip $Z$ with the almost complex structure $\mathfrak{J}_{\mathfrak{p},\sigma}$. Choose $j \in Z$ with $\pi(j)=x$. Recall from \cref{rmk:hinvariance} that the subspace $H_j \subset T_jZ$ is invariant under $\mathfrak{J}_{\mathfrak{p},\sigma}$. Now~\cite[Theorem III]{MR0149505} implies that there exists a pseudo-holomorphic curve $\Sigma\subset (Z,\mathfrak{J}_{\mathfrak{p},\sigma})$ which contains $j$ and has $H_j$ as its tangent space at $j$. Since $H_j\subset T_jZ$ is horizontal, the restriction $\pi^{\prime}_j|_{H_j} : H_j \to T_xM$ is an isomorphism. Therefore, the restriction of $\pi$ to $\Sigma$ is a local diffeomorphism in some neighbourhood of $j$. Hence there exists a neighbourhood $U_x$ of $x \in M$ and a section $[g] : U_x \to Z$ so that $[g](U_x)\subset \Sigma$. Thus, $[g] : U_x \to (Z,\mathfrak{J}_{\mathfrak{p},\sigma})$ is a pseudo-holomorphic curve in the sense of \cref{rmk:defpseudohol}. Taking $\tilde{g}$ to be the unique metric in $[g]$ with volume form $\sigma$ and applying \cref{ppn:pseudohol} shows the claim. The case (i) follows in the same fashion, except that~\cite{MR0149505} is not needed, as $J_{\mathfrak{p}}$ is integrable and hence the construction of a holomorphic curve realising a prescribed $J_{\mathfrak{p}}$-invariant tangent plane is an elementary exercise.  
\end{proof}

\begin{rmk}
Locally we can always find a holomorphic curve $[g] : M \to (Z,J_{\mathfrak{p}})$, but globally this is not always possible. A properly convex projective structure $\mathfrak{p}$ on a closed surface $M$ with $\chi(M)<0$ admits a holomorphic curve $[g] : M \to (Z,J_{\mathfrak{p}})$ if and only if $\mathfrak{p}$ is hyperbolic~\cite{MR4109900}. One would expect that a corresponding global non-existence result should also hold in the pseudo-holomorphic setting for a suitable class of projective surfaces. 
\end{rmk}

\begin{rmk}
If $(M,\mathfrak{p})$ is a closed oriented projective surface of with $\chi(M)<0$, then there exists at most one holomorphic curve $[g] : M \to (Z,J_{\mathfrak{p}})$, see~\cite{MR3384876}. 
\end{rmk}

\begin{rmk}
Hitchin~\cite{MR699802} gave a twistorial construction of (complex) two-di\-men\-sio\-nal holomorphic projective structures. In the holomorphic category such a projective structure corresponds to a complex surface $Z$ having a family of rational curves with self-intersection number one. Denoting the canonical bundle of $Z$ by $K_Z$, such a holomorphic projective surface is metrisable if and only if $K_{Z}^{-2/3}$ admits a holomorphic section which intersects each rational curve in $Z$ at two points~\cite{MR2581355,MR3210600,MR1979367}. 
\end{rmk} 

\begin{rmk}
The notion of a projective structure also makes sense in the complex setting and such structures are referred to as~\textit{c-projective}, see~\cite{arXiv:1512.04516}. Correspondingly, there is a K\"ahler metrisability problem of c-projective structures. Some obstructions to K\"ahler metrisability of a (complex) two-dimen\-sio\-nal c-projective structure have been obtained in~\cite{MR3233113}. 
\end{rmk}
We conclude by describing the holomorphic curves for the standard projective structure $\mathfrak{p}_0$ on the $2$-sphere whose geodesics are the great circles. 
\begin{ex}\label{ex:flat}
Let $S^2$ denote the sphere of radius $1$ centered at the origin in $\R^3$ and $g$ its induced round metric of constant Gauss curvature $1$ whose geodesics are the great circles. We equip $S^2$ with its standard orientation. 

Recall that the unit tangent bundle $\lambda : T_1S^2 \to S^2$ of $(S^2,g)$ carries a canonical coframing $(\omega_1,\omega_2,\psi)$, where $\omega_1,\omega_2$ span the $1$-forms on $T_1S^2$ that are semi-basic for the projection $\lambda$ and $\psi$ denotes the Levi-Civita connection form of $g$. The $1$-forms $(\omega_1,\omega_2,\psi)$ satisfy the structure equations
\begin{equation}\label{eq:strucsphere}
\d\omega_1=-\omega_2\wedge\psi\quad \text{and}\quad \d\omega_2=-\psi\wedge\omega_1\quad \text{and}\quad \d\psi=-\omega_1\wedge\omega_2.
\end{equation}
Let $\hat{g}$ be a Riemannian metric on $S^2$ and write $\lambda^*\hat{g}=\hat{g}_{ij}\omega_i\otimes \omega_j$ for unique real-valued functions $\hat{g}_{ij}=\hat{g}_{ji}$ on $T_1S^2$. Phrased in modern language (c.f.~\cite{MR2581355}) and applied to the case of the $2$-sphere, R.~Liouville's result~\cite{21.0317.01} implies  that if the metrics $\hat{g}$ and $g$  have the same unparametrised geodesics then the functions $h_{ij}:=\hat{g}_{ij}(\hat{g}_{11}\hat{g}_{22}-\hat{g}_{12}^2)^{-2/3}$ satisfy the linear differential equations
\begin{align}\label{eq:difsysliouville}
\begin{split}
\d h_{11}&=-2h_1\omega_2+2h_{12}\psi,\\
\d h_{12}&=h_1\omega_1-h_2\omega_2-(h_{11}-h_{22})\psi,\\
\d h_{22}&=2h_2\omega_1-2h_{12}\psi,
\end{split}
\end{align}
for some smooth real-valued functions $h_i$ on $T_1S^2$. Conversely, a solution to~\eqref{eq:difsysliouville} on $T_1S^2$ satisfying $h_{11}h_{22}-h_{12}^2\neq 0$ gives a Riemannian metric $\hat{g}$ on $S^2 $ with $\lambda^*\hat{g}=(h_{ij}(h_{11}h_{22}-h_{12}^2)^{-2})\omega_i\otimes\omega_j$ and that has the same unparametrised geodesics as $g$. 

Applying the exterior derivative to the above system of equations implies the existence of a unique real-valued function $h$ on $T_1S^2$ such that
\begin{align*}
\d h_1&=-h_{12}\omega_1+(h_{11}+h)\omega_2+h_2\psi,\\
\d h_2&=-(h_{22}+h)\omega_1+h_{12}\omega_2-h_1\psi.
\end{align*}
Taking yet another exterior derivative gives that
\[
\d h=-2h_1\omega_1+2h_2\omega_2.
\]
Writing
\[
\vartheta=\begin{pmatrix} 0 & -\omega_1 & -\omega_2 \\ \omega_1 & 0 & -\psi \\ \omega_2 & \psi & 0 \end{pmatrix}\quad \text{and}\quad H=\begin{pmatrix} h & h_2 & -h_1 \\ h_2 & -h_{22} & h_{12} \\ -h_1 & h_{12} & -h_{11} \end{pmatrix}
\]
the above system of differential equations can be expressed as
\[
\d H+\vartheta H+H \vartheta^t=0. 
\]
The structure equations~\eqref{eq:strucsphere} imply that $\d\vartheta+\vartheta\wedge\vartheta=0$, hence we may write $\vartheta=\Xi^{-1} \d \Xi$ for some diffeomorphism $\Xi : T_1S^2 \to \mathrm{SO}(3)$. It follows that the solutions are of the form $H=\Xi^{-1}C(\Xi^{-1})^t$ for some constant symmetric $3$-by-$3$ matrix $C$. In particular, taking $C=AA^t$ for some $A \in \mathrm{SL}(3,\R)$, we obtain a solution $H_A$ providing a metric $\hat{g}_A$ on $S^2$ having the great circles as its geodesics. 

Finally, in order to construct the holomorphic curve $[\hat{g}_A] : S^2 \to Z$ from $H_A$, we interpret $Z$ as an associated bundle to $T_1S^2$. We will only give a sketch of the construction and refer the reader to~\cite[\S 4]{MR4109900} for additional details. The orientation and metric turn $S^2$ into a Riemann surface and hence a conformal structure on $S^2$ is given in terms of a Beltrami differential. Denoting the canonical bundle of $S^2$ by $K_{S^2}$, a Beltrami differential is a section $\mu$ of $\ov{K_{S^2}}\otimes K_{S^2}^{-1}$ satisfying $|\mu(x)|<1$ for all $x \in S^2$, where $|\cdot|$ denotes the norm induced by the natural Hermitian bundle metric on $\ov{K_{S^2}}\otimes K_{S^2}^{-1}$. The Riemannian metric $g$ gives an isomorphism $\ov{K_{S^2}}\otimes K_{S^2}^{-1}\simeq K_{S^2}^{-2}$ and thus $Z$ may be identified with $T_1S^2\times_{S^1} \mathbb{D}$, where $S^1$ acts by usual rotation on $T_1S^2$ and by $z \cdot \mathrm{e}^{\i\phi}=z\mathrm{e}^{-2\i\phi}$ on the open unit disk $\mathbb{D}\subset \C$. A holomorphic curve $[\hat{g}] : S^2 \to Z$ is therefore represented by a map $\mu : T_1S^2 \to \mathbb{D}$. Explicitly, the conformal structure arising from a Riemannian metric $\hat{g}$ on $S^2$ is represented by the map
\[
\mu=\frac{p-q+2\i r}{p+q+2\sqrt{pq-r^2}},
\]
where we write $\lambda^*\hat{g}=p\omega_1\otimes \omega_1+2r\omega_1\circ \omega_2+q\omega_2\otimes \omega_2$ for unique real-valued functions $p,q,r$ on $T_1S^2$. In our case, the holomorphic curve $[\hat{g}_A] : S^2 \to Z$ is thus represented by $\mu$ with
\[
p=\frac{h_{11}}{(h_{11}h_{22}-h_{12}^2)^2},\qquad r=\frac{h_{12}}{(h_{11}h_{22}-h_{12}^2)^2}, \qquad q=\frac{h_{22}}{(h_{11}h_{22}-h_{12}^2)^2}
\]
and where the functions $h_{ij}$ arise from $H_A$ as above.
\begin{rmk}
In the case of the standard projective structure on $S^2$ the complex surface $(Z,J_{\mathfrak{p}_0})$ is biholomorphic to $\mathbb{CP}^2\setminus\mathbb{RP}^2$ and moreover, the image of a holomorphic curve $[g] : S^2 \to Z$ is a  smooth quadric, see~\cite{MR3144212}. Trying to explicitly relate the holomorphic curve $[\hat{g}_A]$ to its image quadric does in general however not seem to give manageable expressions. 
\end{rmk} 
\end{ex}

\providecommand{\mr}[1]{\href{http://www.ams.org/mathscinet-getitem?mr=#1}{MR~#1}}
\providecommand{\zbl}[1]{\href{http://www.zentralblatt-math.org/zmath/en/search/?q=an:#1}{zbM~#1}}
\providecommand{\arxiv}[1]{\href{http://www.arxiv.org/abs/#1}{arXiv:#1}}
\providecommand{\doi}[1]{\href{http://dx.doi.org/#1}{DOI~#1}}
\providecommand{\href}[2]{#2}

\end{document}